%% file: main.tex
\documentclass[10pt,leqno]{amsart}
\input{Symbols}

\begin{document}

\title[Matrix Harnack Estimates under K\"ahler-Ricci Flow]{Matrix Li-Yau-Hamilton Estimates under K\"ahler-Ricci Flow}

\input{authors}

\begin{abstract}
We prove matrix Li-Yau-Hamilton estimates for positive solutions to the heat equation and the backward conjugate heat equation, both coupled with the K\"ahler-Ricci flow. As an application, we obtain a monotonicity formula.
\end{abstract}

\maketitle

\input{0Introduction}
\input{2Evolution}

\input{3LYHheat}
\input{4LYHbackward}
\input{5Monotonicity}

\input{6Constrained}




\bibliographystyle{alpha}
\bibliography{references}

\end{document}

%% file: Symbols.tex
\usepackage{amsmath}
\usepackage{amsfonts}
\usepackage{amssymb}
\usepackage{graphicx}
\usepackage{color}
\usepackage{hyperref}
\usepackage{xcolor}

\theoremstyle{plain}
\newtheorem{theorem}{Theorem}[section]
\newtheorem{proposition}[theorem]{Proposition}

\theoremstyle{definition} 

\newtheorem*{definition*}{Definition}
\newtheorem*{remark*}{Remark}
\newtheorem{remark}{Remark}[section]

\newcommand{\R}{\mathbb{R}}

\def\p{\partial}
\def\n{\nabla}

\def\a{\alpha}
\def\b{\beta}
\def\d{\delta}
\def\g{\gamma}
\def\k{\kappa}

\def\g{\gamma}

\def\e{\epsilon}
\def\ve{\varepsilon}

\def\ab{\bar{\alpha}}
\def\bb{{\bar{\beta}}}
\def\db{\bar{\delta}}
\def\gb{\bar{\gamma}}

\def\abb{\alpha \bar{\beta}}

\def\Ric{\text{Ric}}

\def\Ric{\operatorname{Ric}}

\numberwithin{equation}{section}
\makeatletter



%% file: authors.tex
\author[X. Li]{Xiaolong Li}\thanks{The first author's research is partially supported by Simons Collaboration Grant \#962228 and a start-up grant at Wichita State University}
\address{Department of Mathematics, Statistics and Physics, Wichita State University, Wichita, KS, 67260, USA}
\email{xiaolong.li@wichita.edu}

\author[H.Y. Liu]{Hao-Yue Liu}\thanks{}
\address{Department of Mathematics, China University of Mining and Technology, Xuzhou, 221116, China}
\email{lhyty7@163.com}

\author[X.A. Ren]{Xin-an Ren}\thanks{The third author's research is partially supported by National Natural Science Foundation of China Grant \#11571361}
\address{Department of Mathematics, China University of Mining and Technology, Xuzhou, 221116, China}
\email{renx@cumt.edu.cn}


\subjclass[2020]{53E30 (Primary), 58J35, 35K05 (Secondary)}

\keywords{Li-Yau-Hamilton estimates, matrix Harnack inequality, heat equation, K\"ahler-Ricci flow}

%% file: 0Introduction.tex
\section{Introduction}

In a fundamental work \cite{LY86}, Li and Yau proved gradient estimates for positive solutions to the heat equation on a Riemannian manifold and derived a sharp version of the classical Harnack inequality of Moser \cite{Moser64} from their estimates. Under stronger curvature assumptions, Hamilton \cite{Hamilton93} extended the estimates of Li and Yau to full matrix estimates on the Hessian. Later on, Chow and Hamilton \cite{CH97} further extended these estimates to the constrained setting and also proved new linear trace Harnack inequalities for the Ricci flow. Analogous estimates were obtained by Cao and Ni \cite{CN05} on K\"ahler manifolds, by Ni \cite{Ni07} under the K\"ahler-Ricci flow, by Yao, Shen, Zhang, and the third author \cite{RYSZ15} in the constrained case. Extensions to a nonlinear heat equation were obtained by X. Cao, Fayyazuddin Ljungberg, and Liu \cite{CFLL13}, Wu \cite{Wu20}, and the third author \cite{Ren19}. 
For more discussions on Li-Yau-Hamilton estimates, we refer the reader to the surveys \cite{Ni08Survey} and \cite{Chow22}, and the monographs \cite[Chapters 15-16]{Chowbookpart2} and \cite[Chapters 23-26]{Chowbookpart3}.

Recently, Zhang and the first author \cite{LZ23} proved new matrix Li-Yau-Hamilton estimates for positive solutions to the heat equation and the backward conjugate heat equation under the Ricci flow. They used these estimates to study the monotonicity of various parabolic frequencies. The purpose of this paper is to prove analogous results under the $\e$-K\"ahler-Ricci flow
\begin{equation}\label{eq eKRF}
\frac{\p}{\p t} g_{\a \bb} =-\epsilon R_{\a \bb}.
\end{equation} 
Here and throughout this paper, $\e$ is a positive constant, $g_{\abb}(t)$ is a family of K\"ahler metrics on a complex manifold $M^m$ of complex dimension $m\geq 1$, and $R_{\abb}$ denotes the Ricci curvature. We will prove matrix Li-Yau-Hamilton estimates for positive solutions to the heat equation 
\begin{equation}\label{eq heat equation}
u_t-\Delta_{g(t)} u=0
\end{equation} 
and the backward conjugate heat equation
\begin{equation}\label{eq backward conjugate heat equation}
u_t+\Delta_{g(t)} u =\epsilon R u,
\end{equation}
where $R$ denotes the scalar curvature, both coupled with the $\e$-K\"ahler-Ricci flow. 
Both \eqref{eq heat equation} and \eqref{eq backward conjugate heat equation} are of fundamental importance in the study of Ricci flow; see Perelman \cite{Perelman1}, Zhang \cite{Zhang06}, Cao and Hamilton \cite{CH09}, Cao and Zhang \cite{CZ11}, and Bamler \cite{Bamler20}.

Our first result states 

\begin{theorem}\label{thm LYH heat}
Let $(M^m,g(t))$, $t\in [0,T]$, be a complete solution to the $\epsilon$-K\"ahler-Ricci flow \eqref{eq eKRF} with nonnegative bisectional curvature and $R_{\a \bb} \leq \k g_{\a \bb}$ for some $\kappa >0$. 
Let $u:M^m \times [0,T] \to \R$ be a positive solution to the heat equation \eqref{eq heat equation}.
Then
\begin{equation}\label{eq matrix LYH heat}
\n_\a \n_\bb \log u + \frac{\epsilon \k}{1-e^{-\epsilon \k t}} g_{\a \bb} \geq 0 
\end{equation} for all $(x,t)\in M\times (0,T)$. 
\end{theorem}

\begin{remark}
Letting $\e \to 0^+$ in Theorem \ref{thm LYH heat} covers the result of Cao and Ni \cite[Theorem 1.1]{CN05}, which asserts that if $u$ is a positive solution to the heat equation on a complete K\"ahler manifold with bounded nonnegative bisectional curvature, then 
\begin{equation}\label{eq Cao-Ni matrix estimate}
    \n_\a \n_{\bar{\b}} \log u +\frac{1}{t}g_{\a\bar{\b}} \geq 0. 
\end{equation}
\end{remark}

\begin{remark}
Cao and Ni \cite[Theorem 3.1]{CN05} proved that if $u$ is a positive solution to the heat equation coupled with a complete K\"ahler-Ricci flow with bounded nonnegative bisectional curvature, then \eqref{eq Cao-Ni matrix estimate} holds provided that $u$ is plurisubharmonic. In contrast to \eqref{eq matrix LYH heat}, their inequality does not depend on the upper bound of Ricci curvature but requires the plurisubharmonicity of $u$. 
The question of whether the plurisubharmonicity of $u$ is preserved by the heat equation, with either static or time-dependent metrics, was investigated by Ni and Tam \cite{NT03AJM, NT03JDG, NT04}. 
\end{remark}

\begin{remark}
Using the method in \cite{Hamilton93} or \cite{LZ23}, a similar result with error terms can be proved for $\epsilon$-K\"ahler-Ricci flows with bisectional curvature bounded from below. 
\end{remark}



Chow and Ni \cite[Theorem 2.2]{Ni07} proved that if $u$ is a positive solution to the forward conjugate heat equation $$u_t-\Delta_{g(t)} u=\e R u,$$ coupled with an $\e$-K\"ahler-Ricci flow with bounded nonnegative bisectional curvature, then 
\begin{equation}\label{eq Ni matrix estimate}
\e R_{\abb} +\n_\a \n_\bb \log u +\frac{1}{t}g_{\abb} \geq 0.
\end{equation}
Our second result provides an analogous result for the backward conjugate heat equation. 

\begin{theorem}\label{thm LYH backward}
Let $(M^m,g(t))$, $t\in [0,T]$, be a complete solution to the $\epsilon$-K\"ahler-Ricci flow \eqref{eq eKRF} with nonnegative bisectional curvature and $R_{\abb} \leq \k g_{\abb}$ for some $\k>0$.
Let $u:M^m \times [0,T] \to \R$ be a positive solution to the backward conjugate heat equation \eqref{eq backward conjugate heat equation}. 
Suppose that $\eta:(0,T) \to (0,\infty)$ is a $C^1$ function satisfying the ordinary differential inequality 
\begin{equation}\label{eq eta(t) ODE}
    \eta' \leq \eta^2 -\e \kappa \eta -\frac{\e \kappa}{t} 
\end{equation}
on $(0,T)$ and that $\eta(t) \to \infty$ as $t\to T$. 
Then
\begin{equation}\label{eq matrix LYH backward}
\e R_{\a \bb} -\n_\a\n_\bb \log u -\eta g_{\a \bb} \leq 0
\end{equation}
for all $(x,t) \in M \times (0,T)$. In particular, we have
\begin{equation}\label{eq LYH backward explicit eta}
\e R_{\a \bb} -\n_\a\n_\bb \log u -\left(\frac{\e \k}{1-e^{-\e \k (T-t)}}+\sqrt{\frac{\k}{t}} \right) g_{\a \bb} \leq 0.
\end{equation}
\end{theorem}

\begin{remark}
In the above two theorems, it suffices to assume $g(0)$ has bounded nonnegative bisectional curvature since this condition is preserved by the $\e$-K\"ahler-Ricci flow. In the compact case, this was proved by Bando \cite{Bando84} for $m=3$ and by Mok \cite{Mok88} for all higher dimensions. The complete noncompact case is due to Shi \cite{Shi90thesis}. 
\end{remark}

Both Chow and Ni's proof of \eqref{eq Ni matrix estimate} and our proof of \eqref{eq matrix LYH backward} make use of Cao's Harnack inequality for the K\"ahler-Ricci flow \cite{Cao92}. On ancient $\e$-K\"ahler-Ricci flows, we have an improved Harnack inequality that leads to the following cleaner estimate. 

\begin{theorem}\label{thm LYH ancient}
Let $(M^m,g(t))$, $t\in (-\infty,T]$, be a complete ancient solution to the $\epsilon$-K\"ahler-Ricci flow \eqref{eq eKRF} with nonnegative bisectional curvature and $R_{\abb} \leq \k g_{\abb}$ for some $\k>0$.
Let $u:M^m \times (-\infty,T] \to \R$ be a positive solution to the backward conjugate heat equation \eqref{eq backward conjugate heat equation}. 
Then
\begin{equation}\label{eq LYH ancient}
\e R_{\a \bb} -\n_\a\n_\bb \log u -\frac{\e \k}{1-e^{-\e \k (T-t)}} \leq 0
\end{equation}
for all $(x,t) \in M \times (-\infty,T)$. 
\end{theorem}

\begin{remark}
In Theorem \ref{thm LYH ancient}, it suffices to assume the weaker condition of nonnegative orthogonal bisectional curvature. Indeed, any ancient solution to the $\e$-K\"ahler-Ricci flow with nonnegative orthogonal bisectional curvature has nonnegative bisectional curvature, as shown in \cite{LN20}. 
\end{remark}

\begin{remark}
Note that nonnegative bisectional curvature is weaker than both nonnegative sectional curvature and nonnegative complex sectional curvature, which are needed to prove the analogous results under the Ricci flow in \cite{LZ23}.
\end{remark}

One of the applications of matrix Li-Yau-Hamilton estimates is to prove the monotonicity of frequency functions. Ni \cite{Ni15} used \eqref{eq Cao-Ni matrix estimate} to prove a monotonicity formula on K\"ahler manifolds with bounded nonnegative bisectional curvature. This was generalized to the K\"ahler-Ricci flow case by Xu and the second author \cite{LX22} using \eqref{eq Ni matrix estimate}. Our estimate \eqref{eq matrix LYH backward} in Theorem \ref{thm LYH backward} also yields a similar monotonicity formula.

\begin{theorem}\label{thm monotonicity}
Let $(M^m,g(t))$, $t\in [0,T]$, be a complete solution to the $\epsilon$-K\"ahler-Ricci flow \eqref{eq eKRF} with nonnegative bisectional curvature and $R_{\abb} \leq \k g_{\abb}$ for some $\k>0$.
Let $f$ be a holomorphic function on $M^m$ of finite order and let $H(x,t):=H(x,t;y,T)$ be the fundamental solution to the backward conjugate heat equation \eqref{eq backward conjugate heat equation} centered at $(y,T)$. Suppose that $\eta:(0,T) \to (0,\infty)$ is a $C^1$ function satisfying \eqref{eq eta(t) ODE}
on $(0,T)$ and that $\eta(t) \to \infty$ as $t\to T$. Then for any $p>0$, the function
\begin{equation*}
e^{\int \eta(t) dt} \cdot \frac{\int_{M} H(x,t)|\nabla f|^{2}(x)|f|^{p-2}(x)d\mu_t}{\int_{M} H(x,t)|f|^{p}(x)d\mu_t}
\end{equation*}
is monotone nonincreasing in $t$ on $[0,T]$.
\end{theorem}

\begin{remark}
Theorem \ref{thm monotonicity} remains valid if one replaces $H(x,t)$ with any positive solution to the backward conjugate heat equation \eqref{eq backward conjugate heat equation}, provided that the integrals are finite and all integration by parts can be justified.
\end{remark}

Finally, we note that both Theorem \ref{thm LYH heat} and its Ricci flow analog in \cite[Theorem 1.1]{LZ23} can be generalized to the constrained setting in the spirit of Chow and Hamilton \cite{CH97}. This is done in Section 6. 

The rest of this paper is organized as follows. In Section 2, we derive the evolution equations that will be used in subsequent sections. In Section 3, we prove Theorem \ref{thm LYH heat}. The proof of Theorem \ref{thm LYH backward} is given in Section 4. Section 5 presents the proof of Theorem \ref{thm monotonicity}. 

%% file: 2Evolution.tex
\section{Evolution Equations}

Let $(M^m,g(t))$, $t\in [0,T]$, be a solution to the $\epsilon$-K\"ahler-Ricci flow \eqref{eq eKRF} and let $u:M^m \times [0,T] \to \R$ be a positive solution to the heat-type equation 
\begin{equation}\label{eq general heat}
(\p_t-\delta \Delta_{g(t)}) u =\theta Ru,
\end{equation}  
where $\d,\theta \in \R$. Note that \eqref{eq general heat} covers both the heat equation \eqref{eq heat equation} (with $\delta=1$ and $\theta=0$) and the backward conjugate heat equation \eqref{eq backward conjugate heat equation} (with $\delta=-1$ and $\theta=\e$). 

In this section, we derive the evolution equation satisfied by $\n_\a \n_{\bb}\log u$, the complex Hessian of $\log u$. For simplicity of notations, we write $\Delta :=\Delta_{g(t)}$, $v:=\log u$, $v_{\a \bb}:=\n_\a \n_{\bb} v$, $v_{\a \g}:=\n_\a \n_\g v$, and $v_{\bb \gb}:=\n_{\bb} \n_{\gb} v$.  Below, $\Delta_L$ denotes the complex Lichnerowicz Laplacian acting on $(1,1)$-tensors via 
\begin{equation*}
    \Delta_L h_{\a \bb} :=\Delta h_{\a \bb} +R_{\a \bb \d \gb} h_{\g \db} -\frac{1}{2}R_{\a \db}h_{\d \bb} -\frac{1}{2}R_{\g \bb}h_{\a \gb},
\end{equation*}

\begin{proposition}\label{prop evolution hessian}
The complex Hessian of $\log u$ satisfies 
\begin{eqnarray*}
 (\p_t-\delta \Delta_L)v_{\a \bb} 
&=& \theta \n_\a \n_\bb R  +\d \left(v_{\a \gb} v_{\g \bb} +v_{\a \g} v_{\bb \gb} +R_{\a \gb  \d \bb} \n_\g v \n_{\db} v \right) \\
&& +\d \left(\n_\g v_{\a \bb} \n_{\gb} v +\n_{\gb} v_{\a \bb} \n_{\g} v \right).
\end{eqnarray*}
\end{proposition}

\begin{proof}
Since $u$ is a positive solution to \eqref{eq general heat}, $v:=\log u$ satisfies
\begin{equation*}
(\p_t -\delta \Delta) v =\delta |\n v|^2 +\theta R.
\end{equation*}
Using the fact that $\p_t$ commutes with $\n_\a\n_\bb$ and the identity
$$\n_\a \n_\bb \Delta f =\Delta_L \n_\a \n_\bb f$$ 
(see for example \cite[page 70 (2.22)]{Chowbookpart1}), 
we calculate that 
\begin{eqnarray*}
(\p_t-\d \Delta_L ) v_{\a \bb}&=& \n_\a \n_\bb (\p_t -\d \Delta) v \\
&=& \theta \n_\a \n_\bb R +\d \n_\a \n_\bb |\n v|^2. 
\end{eqnarray*}
Commuting covariant derivatives produces
\begin{eqnarray*}
 \n_\a \n_\bb |\n v|^2 
&=&  \n_\a \n_\bb (\n_\g v \n_{\gb} v) \\
&=& \n_\a \left(\n_\bb \n_\g v \n_{\gb} v + \n_\g v \n_\bb \n_{\gb}v \right) \\
&=& \n_\a \n_\bb \n_\g v \n_{\gb} v  + \n_\bb \n_\g v  \n_\a  \n_{\gb} v \\
&& +\n_\a  \n_\g v \n_\bb \n_{\gb}v  + \n_\g v \n_\a \n_\bb \n_{\gb}v \\
&=& \n_\g v_{\a \bb} \n_{\gb} v + v_{\a \gb} v_{\g \bb} +v_{\a \g} v_{\bb \gb}  \\
&& + \n_\g v \left(\n_{\gb}v_{\a \bb} +R_{\a \gb  \d \bb} \n_{\db} v \right) \\
&=& v_{\a \gb} v_{\g \bb} +v_{\a \g} v_{\bb \gb} +R_{\a \gb  \d \bb} \n_\g v \n_{\db} v \\
&& + \n_\g v_{\a \bb} \n_{\gb} v +\n_{\gb} v_{\a \bb} \n_{\g} v. 
\end{eqnarray*}
Hence, we obtain
\begin{eqnarray*}
 (\p_t-\delta \Delta_L)v_{\a \bb} 
&=& \theta \n_\a \n_\bb R  +\d \left(v_{\a \gb} v_{\g \bb} +v_{\a \g} v_{\bb \gb} +R_{\a \gb  \d \bb} \n_\g v \n_{\db} v\right) \\
&& +\d \left(\n_\g v_{\a \bb} \n_{\gb} v +\n_{\gb} v_{\a \bb} \n_{\g} v \right).
\end{eqnarray*}
The proof is complete.
\end{proof}

%% file: 3LYHheat.tex
\section{Matrix estimates for the heat equation}
In this section, we prove Theorem \ref{thm LYH heat}.

\begin{proof}[Proof of Theorem \ref{thm LYH heat}]
Setting $\d=1$ and $\theta=0$ in Proposition \ref{prop evolution hessian} yields
\begin{eqnarray*}
(\p_t-\Delta_L)v_{\a \bb} 
&=& v_{\a \gb} v_{\g \bb} +v_{\a \g} v_{\bb \gb} +R_{\a \gb  \d \bb} \n_\g v \n_{\db} v \\
&& +\n_\g v_{\a \bb} \n_{\gb} v +\n_{\gb} v_{\a \bb} \n_{\g} v .
\end{eqnarray*}
Define
\begin{equation*}\label{eq S_abb def}
S_{\a \bb} :=v_{\a \bb} +c(t)g_{\a \bb},
\end{equation*}
where $$c(t)=\frac{\epsilon \k}{1-e^{-\epsilon \k t}}.$$
Using the identity
\begin{equation*}
v_{\a \gb} v_{\g \bb} =S_{\a \gb}S_{\g \bb} -2c(t)S_{\a \bb} +c^2(t)g_{\a \bb},
\end{equation*}
we compute that
\begin{eqnarray*}
(\p_t-\Delta_L)S_{\a \bb} &=& v_{\a \g} v_{\bb \gb} +S_{\a \gb}S_{\g \bb} -2c(t)S_{\a \bb} +R_{\a \gb  \d \bb} \n_\g v \n_{\db} v \\
&& +\n_\g v_{\a \bb} \n_{\gb} v +\n_{\gb} v_{\a \bb} \n_{\g} v \\
&& +c'(t)g_{\a \bb} +c^2(t)g_{\a \bb} -\e c(t) R_{\a \bb} 
\end{eqnarray*}
Using $R_{\a \gb  \d \bb} \n_\g v \n_{\db} v \geq 0$, $R_{\a \bb} \leq \k g_{\a \bb}$, and the equation $$c'(t)=-c^2(t)+\ve \k c(t),$$ 
we obtain that
\begin{eqnarray*}
(\p_t-\Delta_L)S_{\a \bb} &\geq& S_{\a \gb}S_{\g \bb} -2c(t)S_{\a \bb} +\n_\g S_{\a \bb} \n_{\gb} v+\n_{\gb} S_{\a \bb} \n_{\g} v. 
\end{eqnarray*}
Noticing that the two terms $S_{\a \gb}S_{\g \bb}$ and $-2c(t)S_{\a \bb}$ satisfy the null eigenvector condition in Hamilton's tensor maximum principle \cite{Hamilton86} and $S_{\a \bb}(x,t) \to \infty$ uniformly as $t \to 0$, we conclude that if $M^m$ is compact, then $S_{\a \bb}(x,t) \geq 0$ on $M\times (0,T)$.

If $M^m$ is complete noncompact, one can use the same idea as in \cite{Ni07} to justify that the tensor maximum principle in \cite{Ni04} is applicable and conclude the nonnegativity of $S_{\a \bb}$ on $M\times (0,T)$. Alternatively, one can follow the argument in \cite[Section 3]{LZ23} to work with the smallest eigenvalue of $S_{\abb}$ and obtain the conclusion. 

\end{proof}



%% file: 4LYHbackward.tex
\section{Matrix estimates for the backward conjugate heat equation}
In this section, we prove Theorem \ref{thm LYH backward} and Theorem \ref{thm LYH ancient}.

\begin{proof}[Proof of Theorem \ref{thm LYH backward}]
By Proposition \ref{prop evolution hessian} with $\d=-1$ and $\theta=\e$, we have
\begin{eqnarray*}
- (\p_t+\Delta_L)v_{\a \bb} 
&=& -\e \n_\a \n_\bb R  + v_{\a \gb} v_{\g \bb} +v_{\a \g} v_{\bb \gb} +R_{\a \gb  \d \bb} \n_\g v \n_{\db} v \\
&& +\n_\g v_{\a \bb} \n_{\gb} v +\n_{\gb} v_{\a \bb} \n_{\g} v .
\end{eqnarray*}
Under the $\e$-K\"ahler-Ricci flow, we have (see \cite[page 123]{Chowbookpart1})
\begin{equation*}
    \p_t R_{\a \bb} =\e \Delta_LR_{\a \bb} =\e \n_\a \n_{\bb} R,
\end{equation*}
which yields
\begin{equation*}
(\p_t+\Delta_L)(\e R_{\a \bb}) =  \e^2 \left( \Delta R_{\a \bb} +R_{\a \bb \g \db} R_{\d \gb} -R_{\a \gb} R_{\g \bb}\right)+\e \n_\a \n_\bb R.
\end{equation*}
Also, we have
\begin{equation*}
-(\p_t+\Delta_L) (\eta g_{\a \bb}) =-\eta' g_{\a \bb} +\e \eta R_{\a \bb}.
\end{equation*}
Set
\begin{equation*}
    Z_{\a \bb} :=\e R_{\a \bb} -v_{\a \bb} -\eta(t)g_{\a \bb}. 
\end{equation*}
Combing the above evolution equations together, we derive that
\begin{eqnarray*}
&& (\p_t +\Delta_L)Z_{\a \bb} \\
&=&  v_{\a \gb} v_{\g \bb} +v_{\a \g} v_{\bb \gb}+R_{\a \gb  \d \bb} \n_\g v \n_{\db} v +\n_\g v_{\a \bb} \n_{\gb} v +\n_{\gb} v_{\a \bb} \n_{\g} v \\
&& +\e^2 \left( \Delta R_{\a \bb} +R_{\a \bb \g \db} R_{\d \gb} -R_{\a \gb} R_{\g \bb}\right)  -\eta' g_{\a \bb} +\e \eta R_{\a \bb} \\
&=&  v_{\a \gb} v_{\g \bb} +v_{\a \g} v_{\bb \gb}  -\e^2 R_{\a \gb} R_{\g \bb} -\n_\g Z_{\a \bb} \n_{\gb} v -\n_\g v \n_{\gb}Z_{\a \bb} \\
&& + \e^2\left( \Delta R_{\a \bb} +R_{\a \bb \g \db} R_{\d \gb} +\tfrac{1}{\e^2}R_{\a \gb \d \bb}\n_\g v \n_{\db} v \right)  \\
&& +\e^2 \left( \tfrac{1}{\e}\n_\g R_{\a \bb} \n_{\gb} v +  \tfrac{1}{\e} \n_\g v \n_{\gb}R_{\a \bb} +\tfrac{1}{\e t}R_{\a \bb} \right)\\
&& -\eta' g_{\a \bb} +\e \eta R_{\a \bb} -\tfrac{\e}{t}R_{\a \bb} .
\end{eqnarray*}
Using the identity
\begin{eqnarray*}
&& -\tfrac{1}{2}Z_{\a \gb}(\e R_{\g \bb}+v_{\g \bb}-\eta g_{\g \bb}) -\tfrac{1}{2}(\e R_{\a \gb}+v_{\a \gb} -\eta g_{\a \gb})Z_{\g \bb} \\
&=& v_{\a \gb} v_{\g \bb}-\e^2 R_{\a \gb}R_{\g \bb} +2\e \eta R_{\a \bb} -\eta^2 g_{\a \bb},
\end{eqnarray*}
we get 
\begin{eqnarray*}
&& (\p_t + \Delta_L)Z_{\a \bb} \\
&=& v_{\a \g} v_{\bb \gb} + \left(\eta^2-\eta' \right) g_{\a \bb} -\left( \e \eta  +\tfrac{\e}{t} \right) R_{\a \bb} \\
&& -\n_\g Z_{\a \bb} \n_{\gb} v -\n_\g v \n_{\gb}Z_{\a \bb} \\
&&  + \e^2\left( \Delta R_{\a \bb} +R_{\a \bb \g \db} R_{\d \gb} +\tfrac{1}{\e^2}R_{\a \gb \d \bb}\n_\g v \n_{\db} v \right)  \\
&& +\e^2 \left( \tfrac{1}{\e}\n_\g R_{\a \bb} \n_{\gb} v +  \tfrac{1}{\e} \n_\g v \n_{\gb}R_{\a \bb} +\tfrac{1}{\e t}R_{\a \bb} \right)\\
&& -\tfrac{1}{2}Z_{\a \gb}(\e R_{\g \bb}+v_{\g \bb}-\eta g_{\g \bb}) -\tfrac{1}{2}(\e R_{\a \gb}+v_{\a \gb} -\eta g_{\a \gb})Z_{\g \bb}.
\end{eqnarray*}
Now for the $\e$-K\"ahler-Ricci flow, Cao's matrix Harnack estimate \cite{Cao92} with $X_\a =\frac{1}{\e} \n_\a v$ implies that 
\begin{eqnarray*}
   0 &\leq & \Delta R_{\a \bb}+R_{\a \bb \g \db} R_{\d \gb} +\frac{1}{\e^2}R_{\a \gb \d \bb}\n_\g v \n_{\db} v  \\
   && +\frac{1}{\e}\n_\g v \n_{\gb}R_{\a \bb} +\frac{1}{\e}\n_{\gb}v \n_\g R_{\a \bb}+\frac{1}{\e t}R_{\a \bb}.
\end{eqnarray*}
Using $R_{\a \bb} \leq \k g_{\a \bb}$ and \eqref{eq eta(t) ODE}, we have 
\begin{equation*}
 \left(\eta^2-\eta' \right) g_{\a \bb} -\left( \e \eta  +\tfrac{\e}{t} \right) R_{\a \bb}
\geq (\eta^2-\eta' -\e \k \eta -\tfrac{\e \k}{t})g_{\a \bb}
   \geq 0. 
\end{equation*}
Therefore, we have
\begin{eqnarray*}
&&  (\p_t + \Delta_L)Z_{\a \bb} +\n_\g v \n_{\gb} Z_{\a \bb} +\n_{\gb} v \n_\g Z_{\a \bb} \\
& \geq &   -\tfrac{1}{2}Z_{\a \gb}(\e R_{\g \bb}+v_{\g \bb}-\eta g_{\g \bb})  -\tfrac{1}{2}(\e R_{\a \gb}+v_{\a \gb} -\eta g_{\a \gb})Z_{\g \bb}. 
\end{eqnarray*}
Observing that the right-hand side of the above inequality satisfies the null eigenvector condition in Hamilton's tensor maximum principle \cite{Hamilton86} and $Z_{\a \bb}(x,t) \to -\infty$ uniformly as $t \to T$, we conclude that $Z_{\a \bb}(x,t) \leq 0$ on $M\times (0,T)$ if $M^m$ is compact. 
If $M^m$ is complete noncompact, one can use the same idea as in \cite{Ni07} to justify that the tensor maximum principle in \cite{Ni04} is applicable and conclude the nonpositivity of $Z_{\a \bb}$ on $M\times (0,T)$. 

Finally, \eqref{eq LYH backward explicit eta} follows from the observation that the function 
\begin{equation*}
    \eta(t)=\frac{\e \k}{1-e^{-\e \k (T-t)}}+\sqrt{\frac{\k}{t}}
\end{equation*}
satisfies \eqref{eq eta(t) ODE} and $\eta(t) \to \infty$ as $t\to T$. 
\end{proof}

\begin{proof}[Proof of Theorem \ref{thm LYH ancient}]
This is a slight modification of the proof of Theorem \ref{thm LYH backward}. The difference is that on ancient solutions we get an improved Harnack inequality without the term $\frac{1}{\e t}R_{\abb}$.
As a result, we can get rid of the term $\frac{\e \k}{t}$ in \eqref{eq eta(t) ODE} and solve the ODE $\eta'=\eta^2-\e \k \eta$ with $\eta(t)\to \infty$ as $t\to T$ to get 
\begin{equation*}
    \eta(t)=\frac{\e \k}{1-e^{-\e \k (T-t)}}.
\end{equation*}
\end{proof}

%% file: 5Monotonicity.tex
\section{A monotonicity formula}

In this section, we prove the monotonicity formula stated in Theorem \ref{thm monotonicity}. 

\begin{proof}[Proof of Theorem \ref{thm monotonicity}]

Recall that $f$ is a holomorphic function on $M^m$ and $H(x,t):=H(x,t;y,T)$ is the fundamental solution to the backward conjugate heat equation \eqref{eq backward conjugate heat equation} centered at $(y,T)$. Following \cite{Ni15}, we define
\begin{eqnarray*}
Z_{p}(t) &=& \int_{M} H(x,t)|f|^{p}(x)d\mu_t,\\
D_{p}(t) &=& \frac{p}{4}\int_{M} H(x,t)|\nabla f|^{2}(x)|f|^{p-2}(x)d\mu_t.
\end{eqnarray*}
These integrals are finite if we assume that $f$ is of finite order in the sense of Hadamard (see \cite{Ni04}). On K\"ahler manifold, we use the convention that under a normal coordinate, 
$$\Delta =\frac{1}{2}\left(\n_\a \n_{\ab} +\n_{\ab} \n_\a \right)$$ and 
$$\langle \n F,\n G \rangle =\frac{1}{2}\left(\n_\a F \n_{\ab} G+\n_{\ab} F \n_\a G \right).$$ 

Under the $\e$-K\"{a}hler-Ricci flow \eqref{eq eKRF}, the measure $d\mu_t$ with respect to $g(t)$ evolves by $\partial_{t} d\mu_t=-\e Rd\mu$. One calculates
\begin{eqnarray*}
Z_p'(t)  &=& \int_M {(-\Delta H + \e RH){{|f|}^p}} d\mu  - \int_M H{{|f|}^p} (\e R) d\mu_t   \\
    &=& -\int_M {\Delta H{{| f |}^p}} d\mu_t   \\
    &=&  \frac{p}{2}\int_M {\langle {\nabla H,\nabla {{|f|}^2}} \rangle {{|f|}^{p - 2}}} d\mu_t \\
    &=&-\frac{p}{2}\int_M \left( H|\nabla f|^{2}|f|^{p-2}+\left(\frac{p}{2}-1\right)Hf_\a \bar{f}_{\ab}|f|^{p-2}\right)d\mu_t \\
    &=&-pD_{p}(t).
\end{eqnarray*}
A straightforward computation shows that
\begin{eqnarray*}
\frac{4}{p}D_{p}'(t)
&=&\frac{d }{{d t}}\int_M {H{{| {\nabla f} |}^2}{{| f |}^{p - 2}}} d\mu_t  \\
    &=& \int_M {(-\Delta H + \e RH){{| {\nabla f} |}^2}{{| f |}^{p - 2}}} d\mu_t  + \int_M {\e H{R_{\alpha\bar{\beta} }}{f_{\alpha}}{{ \bar{f} }_{\bar{\beta}  }}{{| f |}^{p - 2}}} d\mu_t  \\
    && - \int_M {\e H{{| {\nabla f} |}^2}{{| f |}^{p - 2}}} Rd\mu_t   \\
     &=& -\int_M {\Delta H{{| {\nabla f} |}^2}{{| f |}^{p - 2}}} d\mu_t  + \int_M {\e H{R_{\alpha\bar{\beta} }}{f_{\alpha}}{{\bar{f} }_{\bar{\beta} }}{{| f |}^{p - 2}}} d\mu_t   \\
     &=&   \int_M \left({{H_{\alpha}}{f_{\beta}}{{\bar{f} }_{\bar{\alpha}\bar{\beta} }}{{| f |}^{p - 2}} +\left(\frac{p}{2} - 1\right){H_{\alpha}}{{\bar{f} }_{\bar{\alpha} }}f{{| {\nabla f} |}^2}{{| f |}^{p - 4}}}\right)  d\mu_t \\
     && + \int_M {\e H{R_{\alpha\bar{\beta} }}{f_\alpha}{{\bar{f}}_{\bar{\beta} }}{{| f |}^{p - 2}}} d\mu_t   \\
    &=& -\int_M {{H_{\alpha\bar{\beta} }}{f_\beta}{{\bar{f} }_{\bar{\alpha} }}{{| f |}^{p - 2}}} d\mu_t  + \int_M {\e H{R_{\alpha\bar{\beta} }}{f_{\alpha}}{{\bar{f} }_{\bar{\beta} }}{{| f |}^{p - 2}}} d\mu_t.
\end{eqnarray*}
By Theorem \ref{thm LYH backward}, we have 
\begin{equation*}
    \e R_{\abb}-\n_\a \n_{\bb} \log H -\eta(t) g_{\abb} \leq 0,
\end{equation*}
which implies that 
\begin{equation*}
\e HR_{\abb} -H_{\abb} +H_\a H_{\bb}H^{-1} -\eta(t) H g_{\abb} \leq 0.
\end{equation*}
Using the above estimate, we have 
\begin{eqnarray*}
\frac{4}{p}D_{p}'(t)  &\leq& - \int_M H_\a H_{\bb} H^{-1}f_\b \bar{f}_{\ab} |f|^{p-2}d\mu_t +\eta(t) \int_M H |\n f|^2 |f|^{p-2}  d\mu_t \\
&=& -\int_M |\langle \n H, \n \bar{f} \rangle|^2 H^{-1} |f|^{p-2}d\mu_t +\frac{4}{p}\eta(t) D_p(t).
\end{eqnarray*}
Using $$D_p(t)=-\frac{1}{2}\int_M \langle \n H, \n |f|^2 \rangle |f|^{p-2}d\mu_t,$$
we then obtain
\begin{eqnarray*}
&& D_p'(t)Z_p(t) -Z_p'(t)D_p(t) \\
&\leq & -\frac{p}{4}\left(\int_M H|f|^p d\mu_t \right) \left(\int_M |\langle \n H, \n \bar{f} \rangle|^2 H^{-1} |f|^{p-2}d\mu_t \right) \\
&&+\eta(t) Z_p(t)D_p(t) -\frac{p}{4} \left(\int_M \langle \n H, \n |f|^2 \rangle  |f|^{p-2}d\mu_t \right)^2 \\
&\leq & \eta(t) Z_p(t)D_p(t),
\end{eqnarray*}
where we have used H\"older inequality in the last step. 
It follows that
\begin{equation*}
\frac{d}{dt}\frac{D_p(t)}{Z_p(t)} \leq \eta(t) \frac{D_p(t)}{Z_p(t)},
\end{equation*}
which implies
\begin{equation*}
\frac{d}{dt} \left( e^{\int \eta(t) dt} \cdot \frac{D_p(t)}{Z_p(t)} \right)\leq 0.
\end{equation*}
This proves the desired monotonicity. 

\end{proof}

%% file: 6Constrained.tex
\section{The Constrained Cases}
In this section, we extend Theorem \ref{thm LYH heat} to the constrained case in the spirit of Chow and Hamilton \cite{CH97}. 

\begin{theorem}\label{thm constrained} 
Let $(M^m, g(t))$, $t\in [0,T]$, be a complete solution to the $\epsilon$-K\"{a}hler-Ricci flow \eqref{eq eKRF} with nonnegative holomorphic bisectional curvature and $R_{\abb} \leq \k g_{\abb}$ for some $\k>0$. If $u$ and $\tilde{u}$ are two solutions to the heat equation \eqref{eq heat equation} with $|\tilde{u}|<u$, then we have
$$
\nabla_{\alpha}\nabla_{\bar{\beta}}\log u+\frac{\epsilon \k}{1-e^{-\epsilon \k t}}g_{\alpha\bar{\beta}} \geq \frac{\nabla_{\alpha}h\nabla_{\bar{\beta}}h}{1-h^{2}},
$$
where $h=\frac{\tilde{u}}{u}$.
\end{theorem}

\begin{remark}
Taking $\tilde{u}=c u$ with $|c|<1$, then $\n h \equiv 0$ and Theorem \ref{thm constrained} recovers Theorem \ref{thm LYH heat}. 
\end{remark}

\begin{proof}[Proof of Theorem \ref{thm constrained}]
As before, we write $v=\log u$. By Proposition \ref{prop evolution hessian} with $\d=1$ and $\theta=0$, we have 
\begin{eqnarray*}
(\p_t-\Delta_L)v_{\a \bb} 
&=& v_{\a \gb} v_{\g \bb} +v_{\a \g} v_{\bb \gb} +R_{\a \gb  \d \bb} \n_\g v \n_{\db} v \\
&& +\n_\g v_{\a \bb} \n_{\gb} v +\n_{\gb} v_{\a \bb} \n_{\g} v .
\end{eqnarray*}
Denote by 
$Y_{\abb}=\frac{\nabla_{\alpha}h\nabla_{\bar{\beta}}h}{1-h^{2}}$, $Y_{\a \b}=\frac{\nabla_{\alpha}h\nabla_{\beta}h}{1-h^{2}}$, and $Y_{\ab \bb}=\frac{\nabla_{\ab}h\nabla_{\bar{\beta}}h}{1-h^{2}}$. 
According to \cite[Lemma 3]{RYSZ15}, we have 
\begin{eqnarray*}
 (\p_t-\Delta_L)Y_{\abb} 
&=& -R_{\abb \g \db}Y_{\d \gb}+\nabla_{\gamma}v \nabla_{\gb} Y_{\abb}+\nabla_{\bar{\gamma}}v\nabla_{\gamma}Y_{\abb}+2|\n h|^2 Y_{\abb} \\
&& +v_{\a \g}Y_{\bb \gb} +v_{\a \gb}Y_{\g \bb}+v_{\g \bb}Y_{\a \gb} +v_{\bb \gb}Y_{\a \g} \\
&&-\frac{1}{1-h^{2}}(h_{\a \g}+2hY_{\a \g})(h_{\bb \gb} +2h Y_{\bb \gb}) \\
&&-\frac{1}{1-h^{2}}(h_{\a \gb}+2h Y_{\a \gb})(h_{\g \bb}+2hY_{\g \bb}) \\
&\leq & -R_{\abb \g \db}Y_{\d \gb}+\nabla_{\gamma}v \nabla_{\gb} Y_{\abb}+\nabla_{\bar{\gamma}}v\nabla_{\gamma}Y_{\abb}+2|\n h|^2 Y_{\abb} \\
&& +v_{\a \g}Y_{\bb \gb} +v_{\a \gb}Y_{\g \bb}+v_{\g \bb}Y_{\a \gb} +v_{\bb \gb}Y_{\a \g},
\end{eqnarray*}
where we have used the nonpositivity of the last two lines. 

Set $$P_{\abb}=v_{\abb}-Y_{\abb}.$$ 
A straightforward computation shows that
\begin{eqnarray*}
(\p_t-\Delta_L)P_{\abb} &=& R_{\a \gb  \d \bb} \n_\g v \n_{\db} v + R_{\abb \g \db}\frac{\n_{\d}h \n_{\gb} h}{1-h^2} \\
&&+\nabla_{\gamma}v \nabla_{\gb} P_{\abb}+\nabla_{\bar{\gamma}}v\nabla_{\gamma}P_{\abb}\\
&& +P_{\a \gb}P_{\g \bb}+ (v_{\a \g}-Y_{\a \g})(v_{\ab \gb} -Y_{\ab \gb}).
\end{eqnarray*}
Since the bisectional curvatures are nonnegative, we obtain
\begin{equation*}
    (\p_t-\Delta_L)P_{\abb} \geq  P_{\a \gb}P_{\g \bb}
+\nabla_{\gamma}v \nabla_{\gb} P_{\abb}+\nabla_{\bar{\gamma}}v\nabla_{\gamma}P_{\abb}.
\end{equation*}
Next, we let 
\begin{equation*}
X_{\abb}=P_{\abb}+c(t)g_{\abb}
\end{equation*}
where $c(t)=\frac{\e \k}{1-e^{-\e \k t}}$. Using the equation $c'(t)+c^{2}(t)-\epsilon \k c(t)=0$, we obtain
\begin{eqnarray*}
    (\p_t-\Delta_L)X_{\abb} &\geq&  X_{\a \gb}X_{\g \bb}
+\nabla_{\gamma}v \nabla_{\gb} X_{\abb}+\nabla_{\bar{\gamma}}v\nabla_{\gamma}X_{\abb}-2c(t)X_{\abb}.
\end{eqnarray*}
According to the tensor maximal principle of Hamilton \cite{Hamilton86}, this completes the proof of the Theorem in the compact case. In the complete noncompact case, the tensor maximum principle in \cite{Ni04} can be applied and we get the desired conclusion.

\end{proof}

The Ricci flow analog of Theorem \ref{thm constrained} states 
\begin{theorem}\label{thm constrained RF}
Let $(M^n, g(t))$, $t\in [0,T]$, be a complete solution to the Ricci flow 
$$\p_t g=-2\Ric$$ on a manifold $M^n$ of real dimension $n$. 
Let $u, \tilde{u} :M^n \times [0,T] \to \R$ be two solutions to the heat equation \eqref{eq heat equation} satisfying $|\tilde{u}|<u$. 
Suppose that $(M^n,g(t))$ has nonnegative sectional curvature and $\Ric\leq \kappa g$ for some constant $\kappa>0$. Then for all $(x,t)\in M\times (0,T)$, 
\begin{equation*}
        \n_i \n_j 
        \log u + \frac{\kappa }{1-e^{-2\kappa t}} g_{ij} \geq \frac{\n_i h \n_j h}{1-h^2},
\end{equation*}
where $h=\frac{\tilde{u}}{u}$. 
\end{theorem}

\begin{remark}
Theorem \ref{thm constrained RF} covers Theorem 1.1 in \cite{LZ23} by taking $\tilde{u}=c u$ with $|c|<1$. 
\end{remark}

\begin{proof}[Proof of Theorem \ref{thm constrained RF}]
Let $v=\log u$ and 
\begin{equation*}
    P_{ij}=\n_i \n_j v-\frac{\n_i h \n_j h}{1-h^2}. 
\end{equation*}
By a similar calculation as in \cite{CH97}, one gets 
\begin{eqnarray*}
&& (\p_t -\Delta_L) P_{ij} \\
&=&2P_{ik}P_{jk} +2R_{ikjl}\n_k v \n_l v +2R_{ikjl}\frac{\n_k h \n_l h}{1-h^2}+2\n_k v \n_kP_{ij} \\
&& +\frac{2}{1-h^2}\left(\n_i \n_k h +\frac{2h\n_i h \n_k h}{1-h^2} \right) \left(\n_j  \n_k h+\frac{2h\n_j h \n_k h}{1-h^2} \right)
\end{eqnarray*}
The only difference from \cite[Lemma 3.4]{CH97} is that we do not have the terms involving the covariant derivatives of $\Ric$. Noticing the last term is nonnegative, we use the nonnegative sectional curvature condition to conclude
\begin{equation*}
(\p_t -\Delta_L) P_{ij} \geq  2P_{ik}P_{jk} +2\n_k v \n_kP_{ij}.
\end{equation*}
Set 
\begin{equation*}
    Q_{ij}=P_{ij}+c(t)g_{ij},
\end{equation*}
where
\begin{equation*}\label{eq c(t) def}
    c(t)=\frac{\kappa }{1-e^{-2\kappa t}}.
\end{equation*}
Using the identity
\begin{equation*}
    P_{ik}P_{jk} =Q_{ik}Q_{jk}-2c(t)Q_{ij}+c^2(t)g_{ij},
\end{equation*}
we get
\begin{eqnarray*}
(\p_t -\Delta_L) Q_{ij} &\geq & 2Q_{ik}Q_{jk} -4c(t)Q_{ij} +2\n_k v \n_k Q_{ij} \\
&& +c'(t)g_{ij}+2c^2(t)g_{ij}-2c(t) R_{ij}.
\end{eqnarray*}
In view of $R_{ij}\leq \k g_{ij}$ and $c'(t)+2c^2(t)-2\k c(t)=0$, we obtain
\begin{equation*}
(\p_t -\Delta_L) Q_{ij} = 2Q_{ik}Q_{jk} -4c(t)Q_{ij} +2\n_k v \n_k Q_{ij}.
\end{equation*}
If $M^n$ is compact, we can apply Hamilton's tensor maximum principle \cite{Hamilton86} to conclude that $Q_{ij}\geq 0$ on $M\times (0,T)$. If $M^n$ is complete noncompact, we can follow the argument in \cite{LZ23} to complete the proof. 
\end{proof}